\documentclass[preprint,12pt]{elsarticle1}

\usepackage{subfig}
\usepackage{graphicx}
\usepackage{caption,color}   
\usepackage{amssymb}
\usepackage{amsthm}
\usepackage{amsmath}
\usepackage{epic}
\usepackage{setspace}
\usepackage{float}
\usepackage{multirow}
\newtheorem{thm}{Theorem}[section]

\newtheorem{lem}[thm]{Lemma}

\newtheorem{state}[thm]{Statement}

\newenvironment{proofb}{\begin{proof}[Proof of Theorem 1.2]}{\end{proof}}
\numberwithin{equation}{section}
\journal{}

\begin{document}
\begin{spacing}{1.15}
\begin{frontmatter}
\title{\textbf{\\All eigenvalues of the power hypergraph and signed subgraphs of a graph }}
\author[label1,label2]{Lixiang Chen}\ead{chenlixiang@hrbeu.edu.cn}
\author[label2]{Edwin R. van Dam}\ead{Edwin.vanDam@tilburguniversity.edu}
\author[label1]{Changjiang Bu}\ead{buchangjiang@hrbeu.edu.cn}
\address[label1]{College of Mathematical Sciences, Harbin Engineering University, Harbin, PR China}
\address[label2]{Department of Econometrics and O.R., Tilburg University, Tilburg, The Netherlands}

\begin{abstract}
We show how all eigenvalues of a power hypergraph $G^{(k)}$ can be generated from the eigenvalues of signed subgraphs of the underlying graph $G$.
This fixes an incorrect statement in the case of power hypergraphs from [Linear Algebra and its Applications, 590:243-257, 2020].
\end{abstract}

\end{frontmatter}

\section{Introduction}
A hypergraph $H$ is called \emph{$k$-uniform} if each edge of $H$ contains exactly $k$ vertices.
The eigenvalues of the adjacency tensor of $H$ are called the eigenvalues of $H$ \cite{cooper2012spectra}.
The $k$-power hypergraph $G^{(k)}$ is the $k$-uniform hypergraph that is obtained by adding $k-2$ new vertices to each edge of a graph $G$, for $k \geq 3$ (where each edge of $G$ gets different new vertices).

Zhou, Sun, Wang, and Bu \cite[Thm.~16]{Zhou2014Some} showed that the complex solutions $\lambda$ of $\lambda^k=\beta^{2}$ are eigenvalues of $G^{(k)}$ if $\beta$ is an eigenvalue of $G$, and that also the spectral radius of $G^{(k)}$ can be obtained this way. Moreover, they showed that also the eigenvalues of subgraphs of $G$ give rise to eigenvalues of $G^{(k)}$, for $k \geq 4$.

In \cite[Thm.~3]{cardoso2020spectrum}, it was stated that all distinct eigenvalues of a (so-called) generalized power hypergraph $H_s^{(k)}$ can be generated from eigenvalues of subgraphs of the $r$-uniform hypergraph $H$.
When we restrict that result to the case that $H$ is a graph $G$ (i.e., $r=2$), then we obtain the following statement.

\begin{state}\label{chenshu1}
The complex number $\lambda$ is an eigenvalue of $G^{(k)}$ if and only if \\
\noindent(a) some induced subgraph of $G$ has an eigenvalue $\beta$ such that $\beta^2=\lambda^k$, when $k=3$;\\
\noindent(b) some subgraph of $G$ has an eigenvalue $\beta$ such that $\beta^2=\lambda^k$, when $k\geq4$.
\end{state}

However, this result is incorrect, as we shall see in Section \ref{sec:3.1}.
In this paper, we shall extend the result of Zhou et al.~\cite{Zhou2014Some}, and fix the above incorrect statement by using the spectra of signed subgraphs of $G$.

A signed graph $G_\pi$ is a pair $(G, \pi)$, where $G=(V,E)$ is a graph and $\pi:E \rightarrow \{+1,-1\}$ is the edge sign function.
We use $i \sim j$ to denote that the vertices $i$ and $j$ are adjacent in the graph $G$.
The adjacency matrix $A(G_\pi)=(A_{ij})$ of the signed graph $G_\pi$ is the symmetric $\{0,+1,-1\}$-matrix, where
\begin{align*}
A_{ij}=\left\{ \begin{array}{l}
 \pi(i,j),{\kern 35pt}  i \sim j, \\
 0, {\kern 57pt}\mathrm{ otherwise}. \\
\end{array} \right.
\end{align*}
The eigenvalues of $A(G_{\pi})$ are called the eigenvalues of $G_{\pi}$.
An (induced) subgraph of the signed graph $G_{\pi}$ is called a signed (induced) subgraph of $G$.
Using the eigenvalues of signed (induced) subgraphs of $G$,  we can obtain all distinct eigenvalues of $G^{(k)}$ as follows.

\begin{thm}\label{dingli2}
The complex number $\lambda$ is an eigenvalue of $G^{(k)}$ if and only if \\
\noindent(a) some\textbf{ signed induced subgraph} of $G$ has an eigenvalue $\beta$ such that $\beta^2=\lambda^k$, when $k=3$;\\
\noindent(b) some \textbf{signed subgraph} of $G$ has an eigenvalue $\beta$ such that $\beta^2=\lambda^k$, when $k\geq4$.
\end{thm}

We emphasize that the theorem states that every $k$-th root of $\beta^2$ is an eigenvalue of $G^{(k)}$. On the other hand, it does not state that both square roots of $\lambda^k$ are eigenvalues of a signed subgraph. However, if $\beta$ is an eigenvalue of the signed subgraph $\widehat{G}_{\pi}$, then $-\beta$ is an eigenvalue of  $\widehat{G}_{-\pi}$.

The rest of this paper is organized as follows. In Section \ref{zhunbeizhangjie},  some notation and basic definitions are introduced.
In Section \ref{zhuyaojieguo}, we will first extend the result of Zhou et al.~\cite{Zhou2014Some}, then give a counterexample to Statement \ref{chenshu1}, and finish with the proof of Theorem \ref{dingli2}.
\section{Preliminaries}\label{zhunbeizhangjie}
For a positive integer $n$, let $\left[ n \right] = \left\{ {1, \ldots ,n} \right\}$.
A $k$-order $n$-dimensional complex tensor $T= \left( {{t_{{i_1} \cdots {i_k}}}} \right) $ is a multidimensional array with $n^k$ entries over the complex number field $\mathbb{C}$, where ${i_j} \in \left[ n \right]$, for $j = 1, \ldots ,k$.
For $\mathbf{x} = {\left( {{x_1}, \ldots ,{x_n}} \right)^{\top}} \in {\mathbb{C}^n}$, we define ${\mathbf{x}^{\left[ {k - 1} \right]}} = {\left( {x_1^{k - 1}, \ldots ,x_n^{k - 1}} \right)^{\top}}$.
Moreover, $T{\mathbf{x}^{k - 1}}$ denotes a vector in $\mathbb{C}^{n}$ whose $i$-th component is
\[{\left( {T{\mathbf{x}^{k - 1}}} \right)_i} = \sum\limits_{{i_2}, \ldots ,{i_k}=1}^n {{t_{i{i_2} \cdots {i_k}}}{x_{{i_2}}} \cdots {x_{{i_k}}}} .\]
If there exists a nonzero vector $\mathbf{x} \in {\mathbb{C}^n}$ such that $T{\mathbf{x}^{k - 1}} = \lambda {\mathbf{x}^{\left[ {k - 1} \right]}}$, then $\lambda \in \mathbb{C}$ is called an \emph{eigenvalue} of $T$ and $\mathbf{x}$ is an \emph{eigenvector} of $T$ corresponding to $\lambda$.
The pair $(\lambda,\mathbf{x})$ is called an eigenpair of $T$ \cite{lim2005singular,qi2005eigenvalues}.

A hypergraph $H=(V,E)$ is called \emph{$k$-uniform} if each edge of $H$ contains exactly $k$ vertices.
Similar to the relation between graphs and matrices, there is a natural correspondence between uniform hypergraphs and (symmetric) tensors.
Indeed, for a $k$-uniform hypergraph $H$ with $n$ vertices, its \emph{adjacency tensor} ${A}_H=(a_{i_1i_2\ldots i_k})$ is a $k$-order $n$-dimensional tensor, where
\[{a_{{i_1}{i_2} \ldots {i_k}}} = \left\{ \begin{array}{l}
 \frac{1}{{\left( {k - 1} \right)!}},{\kern 37pt}{ \left\{ {{i_1},{i_2},\ldots ,{i_k}} \right\} \in {E}}, \\
 0, {\kern 57pt}\mathrm{ otherwise}. \\
 \end{array} \right.\]

For a graph $G=(V(G),E(G))$ and $e \in E(G)$, we use $N_{e}$ to denote the set of added vertices of $G^{(k)}$ on the edge $e$.
Thus, the set $e\cup N_e$ is a hyperedge of $G^{(k)}$.
By $E_i(G^{(k)})$, we denote the set of hyperedges containing $i$.

Let $x^{S}=\prod_{s \in S}x_s$ for $S\subseteq V(G^{(k)})$.
Then it follows easily that $(\lambda,\mathbf{x})$ is an eigenpair of $G^{(k)}$
if and only if
\begin{align}\label{shizi1}
\lambda x^{k-1}_i&=\sum_{h \in E_i(G^{(k)})}{x^{h\setminus \{i\}}}\notag\\
&=\sum_{j \sim i}{x_jx^{N_{\{i,j\}}}}
\end{align}
for every $i \in V(G)$ and
\begin{align}\label{shizi2}
  \lambda x^{k-1}_v = x_ix_jx^{N_{\{i,j\}}\setminus \{v\}}
\end{align}
for every $v \in N_{\{i,j\}}$ and $\{i,j\} \in E(G)$.

\section{All eigenvalues of the power hypergraph}\label{zhuyaojieguo}

In this section, we will prove our main result and give a counterexample to Statement \ref{chenshu1}.

\subsection{More eigenvalues from signed subgraphs and a counterexample to Statement \ref{chenshu1}}\label{sec:3.1}

First, we will extend the result of Zhou et al.~\cite{Zhou2014Some} by showing how to obtain more eigenvalues of the power hypergraph by using signed subgraphs.

Let $(\beta, \mathbf{y})$ be an eigenpair of some signed (induced, if $k=3$) subgraph $\widehat{G}_{\pi}$ of $G$, with $\beta \neq 0$.
Let $\lambda \in \mathbb{C}$ be such that $\lambda^k=\beta^2$.
We now let $\gamma \in \mathbb{C}$ be such that $\gamma^2=\lambda$, then $\gamma^{2k}=\lambda^k=\beta^2$, and hence $\gamma^k=\pm\beta$.
Since both $(\beta,\mathbf{y})$ and $(-\beta,\mathbf{y})$ are eigenpairs of some signed subgraphs ($\widehat{G}_{\pi}$ and $\widehat{G}_{-\pi}$), we assume without loss of generality that $\gamma^k=\beta$.

For $i \in V(\widehat{G})$, we let $z_i$ be any $k$-th root of $y_i$ and for $\{i,j\} \in E({\widehat{G}})$, we let $\pi_k(i,j)$ be any $k$-th root of $\pi(i,j)$.

For each $\{i,j\}\in E({\widehat{G}})$, we also fix $v_{ij}$ as one of the vertices in $N_{\{i,j\}}$. Using these, we can construct an eigenvector as follows.
Let $\mathbf{x}$  be the vector with entries
\begin{align*}
x_v=\left\{ \begin{array}{ll}
z_i^2, & \mathrm{for}~v=i \in V({\widehat{G}}),\\
\pi_k(i,j)^3z_iz_j/\gamma, & \mathrm{for}~ v=v_{ij} \in N_{\{i,j\}} ~ \mathrm{and} ~\{i,j\} \in E({\widehat{G}}),\\
\pi_k(i,j)z_iz_j/\gamma, &\mathrm{for}~ v\in N_{\{i,j\}}\setminus \{v_{ij}\} ~ \mathrm{and} ~\{i,j\} \in E({\widehat{G}}),\\
0, &\mathrm{otherwise}.
\end{array} \right.
\end{align*}

\begin{lem}\label{xinyinli}

The pair $(\lambda, \mathbf{x})$ is an eigenpair of $G^{(k)}$.
\end{lem}

\begin{proof}
First, we note that for vertices not in the power graph of the subgraph, the required equations \eqref{shizi1} and \eqref{shizi2} are trivial. The same holds for vertices $i \in V(G)$ for which $y_i=0$ ($z_i=0$) and the related vertices in $N_{\{i,j\}}$. In the remainder of the proof, we may therefore assume that $x_i \neq 0$ for $i \in V({\widehat{G}})$ and $x_v \neq 0$ for $v \in N_{\{i,j\}}$.

Next, for each $\{i,j\}\in E({\widehat{G}})$ and each $v \in N_{\{i,j\}}$, we have that
$$\lambda x_v^k=\lambda [\pi_k(i,j)\frac{z_iz_j}{\gamma}]^k=z_i^2z_j^2\pi_k(i,j)^k\frac{(z_iz_j)^{k-2}}{\gamma^{k-2}}=x_ix_jx^{N_{\{i,j\}}},$$
(this also holds for $v=v_{ij}$, because $\pi(i,j)^3=\pi(i,j)$), which shows \eqref{shizi2}.

Finally, for $i \in V(\widehat{G})$, we obtain from the above that
$$\lambda x_i^k=\lambda y_i^2=\frac{\lambda y_i}{\beta} \sum_{j \sim i}\pi(i,j)y_j=\sum_{j \sim i}\lambda [\pi_k(i,j)\frac{z_iz_j}{\gamma}]^k=\sum_{j \sim i}x_ix_jx^{N_{\{i,j\}}},$$
which shows \eqref{shizi1}.
\end{proof}

Note that for odd $k$, one can replace the factors $\pi_k(i,j)^{3}$ and $\pi_k(i,j)$ in the definition of $\mathbf{x}$ by $\pi(i,j)$, to obtain a somewhat simpler expression.

From Lemma \ref{xinyinli}, we can easily get a counterexample for Statement \ref{chenshu1}. Indeed, let $K_4$ be the complete graph with four vertices, and consider the signed subgraph $K_4^{-}$ by signing one of its edges, say $\{1,2\}$, negative. This signed subgraph has eigenvalue $\sqrt{5}$ with eigenvector $(\sqrt{5}-1,\sqrt{5}-1,2,2)^{\top}$, as one can easily check. Thus, by Lemma \ref{xinyinli}, the power hypergraph $K^{(3)}_4$ has an eigenvalue $\sqrt[3]{5}$. 
According to Statement \ref{chenshu1}, some induced subgraph of $K_4$ should therefore have an eigenvalue $\pm \sqrt{5}$.
But this is clearly not the case, because the induced subgraphs are complete graphs, which only have integer eigenvalues.

\subsection{Characterizing all eigenvalues of the power hypergraph}

To finish, we will show that each eigenvalue of a power hypergraph must be obtained from a signed subgraph, thus proving Theorem \ref{dingli2}.

\begin{proofb}
First of all, it is known that a $k$-uniform hypergraph always has an eigenvalue $0$ for $k\geq 3$ \cite{Qi2014Heigenvalue}.
On the other hand, $0$ occurs as an eigenvalue of the induced subgraphs $K_1$. Thus, for the remainder of the proof, we only need to consider the case of $\lambda \neq 0$.

Clearly, Lemma \ref{xinyinli} shows one implication of Theorem \ref{dingli2}. Thus, what remains to show is that every nonzero eigenvalue $\lambda$ of the power hypergraph gives rise to an (appropriate) eigenvalue $\beta$ of a signed (induced) subgraph.

Let $(\lambda,\mathbf{x})$ be an eigenpair of $G^{(k)}$ with $\lambda\neq 0$. Let $\beta$ be such that $\beta^2=\lambda^k$, and let $y_i$ be such that $y_i^2=x_i^k$.
Consider the induced subgraph $\widehat{G}$ on the vertices $ i \in V(G)$ with $x_i \neq 0$. Note that $x_i=0$ for all $i \in V(G)$ is impossible, because of \eqref{shizi2} and $\lambda \neq 0$.
We will indeed prove that there is a signed subgraph on $\widehat{G}$ with an eigenvalue $\beta$ and eigenvector $\mathbf{y}$.

Indeed, using \eqref{shizi2}, we have that
\begin{align*}
  \lambda^{k-2}(x^{N_{\{i,j\}}})^{k-1}&=\prod_{v\in {N_{\{i,j\}}}}\lambda x_v^{k-1}\\
  &=\prod_{v\in {N_{\{i,j\}}}}x_ix_jx^{N_{\{i,j\}}\setminus \{v\}}\\
  &=(x_ix_j)^{k-2}(x^{N_{\{i,j\}}})^{k-3},
\end{align*}
that is,
$(x^{N_{\{i,j\}}})^{k-3}\left(\lambda^{k-2}(x^{N_{\{i,j\}}})^{2}-(x_ix_j)^{k-2}\right)=0$,
and hence
\begin{align*}
(\beta x_ix_jx^{N_{\{i,j\}}})^{k-3}\left((\beta x_ix_jx^{N_{\{i,j\}}})^{2}-(\lambda y_iy_j)^2\right)=0.
\end{align*}
Therefore, we have that $\beta x_ix_jx^{N_{\{i,j\}}}= \mathrm{sgn}(i,j)(\lambda y_iy_j)$, where
\begin{align}\label{signing}
\mathrm{sgn}(i,j) \in \left\{ \begin{array}{l}
\{\pm1\},{\kern 27pt} k=3,\\
\{\pm1,0\} ,{\kern 17pt} k\geq 4. \\
 \end{array} \right.
\end{align}
So, using \eqref{shizi1}, we obtain that for every $i \in V(\widehat{G})$,
\begin{align}\label{eigenequation}
\beta y_i=\frac{\beta x_i}{\lambda y_i}\lambda x_i^{k-1}=\frac{\beta x_i}{\lambda y_i} \sum_{j \sim i}x_jx^{N_{\{i,j\}}}=\frac{1}{\lambda y_i} \sum_{j \sim i}\beta x_ix_jx^{N_{\{i,j\}}}=\sum_{j:\{i,j\}\in E(\widehat{G})}\mathrm{sgn}(i,j)y_j.
\end{align}
Now note that \eqref{signing} defines a sign function on a subgraph of $\widehat{G}$ (or on $\widehat{G}$ itself for $k=3$), and \eqref{eigenequation} shows that this signed subgraph has eigenvalue $\beta$.
\end{proofb}

\section*{Acknowledgement}
The authors would like to thank the referees for their helpful comments.
And this work is supported by the National Natural Science Foundation of China (No. 12071097), the Natural Science Foundation for The Excellent Youth Scholars
of the Heilongjiang Province (No. YQ2022A002) and the Fundamental Research Funds for the Central Universities.

\section*{References}
\bibliographystyle{plain}
\bibliography{spbib}
\end{spacing}
\end{document}